\documentclass[a4paper,10pt,leqno]{amsart}
        \title{The Farrell-Hsiang method revisited}
       \author{Bartels, A.}
       \address{Westf\"alische Wilhelms-Universit\"at M\"unster\\
               Mathematisches Institut\\
               Einsteinstr.~62,
               D-48149 M\"unster, Germany}
        \email{bartelsa@math.uni-muenster.de}
        \urladdr{http://www.math.uni-muenster.de/u/bartelsa} 
        \author{L\"uck, W.}
        \address{Mathematisches Institut der Universit\"at Bonn\\
                Endenicher Allee 60\\
                53115 Bonn, Germany}
         \email{wolfgang.lueck@him.uni-bonn.de}
         \urladdr{http://www.him.uni-bonn.de/lueck}
         \date{July 2011}

  \usepackage{pdfsync}
  \usepackage{calc}
  \usepackage{enumerate,amssymb}
  \usepackage[arrow,curve,matrix,tips,2cell]{xy}
    \SelectTips{eu}{10} \UseTips
    \UseAllTwocells
  \usepackage{tikz}
  \usepackage{pdfcolmk}

  \DeclareMathAlphabet{\matheurm}{U}{eur}{m}{n}

  \newcommand{\IN}{\mathbb{N}}

  \newcommand{\IZ}{\mathbb{Z}}

  \newcommand{\cala}{\mathcal{A}}
  \newcommand{\calb}{\mathcal{B}}

  \newcommand{\calf}{\mathcal{F}}
  \newcommand{\calg}{\mathcal{G}}
  \newcommand{\calh}{\mathcal{H}}

  \newcommand{\calo}{\mathcal{O}}

  \newcommand{\bfK}{{\mathbf K}}
  \newcommand{\bfL}{{\mathbf L}}

  \newcounter{commentcounter}

  \newcommand{\ingreen}[1]{\textcolor{green}{#1}}

  \theoremstyle{plain}
  \newtheorem{theorem}{Theorem}[section]
  \newtheorem{lemma}[theorem]{Lemma}
  
  \newtheorem{proposition}[theorem]{Proposition}

  \theoremstyle{definition}
  \newtheorem{definition}[theorem]{Definition}

  \newtheorem*{definition*}{Definition}

  \theoremstyle{remark}
  \newtheorem{remark}[theorem]{Remark}

  \makeatletter\let\c@equation=\c@theorem\makeatother

   {\end{list}}

  \DeclareMathOperator{\Hom}{Hom}
  \DeclareMathOperator{\id}{id}
  
  \DeclareMathOperator{\ind}{ind}
  \DeclareMathOperator{\mor}{mor}

  \DeclareMathOperator{\res}{res}
  \DeclareMathOperator{\tr}{tr}
  
  \DeclareMathOperator{\Sw}{Sw}
  \DeclareMathOperator{\GL}{GL}
  \DeclareMathOperator{\GW}{GW}
  
  \DeclareMathOperator{\VCyc}{VCyc}

  \newcommand{\modules}{{\matheurm{mod}}}

  \newcommand{\x}{{\times}}
  \newcommand{\ox}{{\otimes}}

  \newcommand{\ol}[1]{{\overline{#1}}}
  \newcommand{\wt}[1]{{\widetilde{#1}}}
  \newcommand{\einsu}{[1,\infty)} 

\begin{document}

  \maketitle

  \begin{abstract}
    We present a sufficient condition for groups
    to satisfy the Farrell-Jones Conjecture in algebraic $K$-theory
    and $L$-theory.
    The condition is formulated in terms of finite quotients of
    the group in question and is motivated by work of Farrell-Hsiang. 
  \end{abstract}

  \ingreen{This version is different from the published version.
    A number of typos and an incorrect formula for the
    transfer before Lemma~\ref{lem:tilde-tr} pointed out by Holger Reich
    have been corrected.}


\typeout{------------------- Introduction -----------------}

  \section*{Introduction}
  
  Farrell-Hsiang used in~\cite{Farrell-Hsiang(1978b)} a beautiful 
  combination of
  controlled topology and Frobenius induction to prove that the 
  Whitehead group
  of fundamental groups of compact flat Riemannian manifolds is trivial.  
  This general method has been refined and used further by Farrell-Hsiang,
  Farrell-Jones and Quinn, see for
  example~\cite{Farrell-Hsiang(1981b),Farrell-Hsiang(1983),
    Farrell-Jones(1988b),Quinn(2005)}.  
  These results belong to a much wider
  collection of results that ultimately led to the Farrell-Jones 
  Conjecture~\cite{Farrell-Jones(1993a)} that predicts a  formula for
  $K$- and $L$-theory of group rings $RG$.
  This formula describes these groups in terms of group homology
  and $K$- and $L$-theory of group rings $RV$, where $V$
  varies over the family $\VCyc$ of virtually cyclic subgroups of $G$.
  Often it is useful to consider a variant of the Conjecture
  where $\VCyc$ is replaced by a larger families of subgroups. 
  For more information about the
  Farrell-Jones Conjecture and its applications we refer
  to~\cite{Bartels-Lueck-Reich(2008appl), Lueck-Reich(2005)}.

  The present paper gives an axiomatic treatment of the 
  Farrell-Hsiang method
  leading us to the definition of Farrell-Hsiang groups below.  
  More generally we define a group to be a Farrell-Hsiang 
  group with respect to a given family
  of subgroups $\calf$, more or less if the Farrell-Hsiang method is 
  applicable relative to $\calf$.  
  Our main result states that the
  Farrell-Jones Conjecture holds for these groups relative to the 
  family $\calf$.  
  In the most important case
  $\calf$ is the family $\VCyc$ of virtually cyclic subgroups
  or a family of groups for which the Farrell-Jones Conjecture relative to 
  $\VCyc$ is known.
  In this case our
  result implies that if a group $G$ is a Farrell-Hsiang group relative to
  $\calf$, then $G$ satisfies both the $K$- and $L$-theoretic 
  Farrell-Jones Conjecture with coefficients in additive categories. 
  Our main result here is used in  work with Tom Farrell 
  to prove the Farrell-Jones Conjecture for virtually poly-cylic 
  groups~\cite{Bartels-Farrell-Lueck(2011)}.
  We give a very brief  overview of this application in
  an Appendix where we also discuss examples of
  Farrell-Hsiang groups.

  In~\cite{Farrell-Jones(1987)} Farrell-Jones used a wonderful
  combination of controlled topology and
  the dynamics of the geodesic flow on negatively curved manifolds
  to prove that the Whitehead group of the fundamental group of
  such manifolds vanishes.
  This Farrell-Jones method has also been refined and further
  used in many papers about the Farrell-Jones conjecture and
  the Borel conjecture, 
  see for example~\cite{Farrell-Jones(1993a), Farrell-Jones(1993c)}.
  In~\cite{Bartels-Lueck-Reich(2008hyper), Bartels-Lueck(2009borelhyp)}
  an axiomatic treatment for this method is given 
  that is from a formal point of view very similar to our treatment here.
  In both cases a transfer and a contracting map are the main ingredients.
  The main difference is, that the transfer in the Farrell-Hsiang
  method uses a finite discrete fiber and its construction 
  depends on Frobenius
  induction, whereas in the Farrell-Jones method the fiber is a compact 
  contractible space and the transfer is essentially 
  given by the tensor product with the singular chain complex of this fiber.
  Also, in applications the construction of the contracting maps 
  is very different.
  In the first case subgroups of finite but large index are exploited,
  in the second case the dynamic of flow spaces is a key ingredient.   

  \subsection*{Acknowledgements}
  The first author thanks Frank Quinn for a long email exchange 
  about the Farrell-Hsiang
  method.  
  This paper was supported by the SFB 878 \emph{Groups, Geometry and Actions}
  and by the Leibniz-award of the second author.


\typeout{-------------- Section 1: Farrell-Hsiang groups  -----------------}

  \section{Farrell-Hsiang groups}

  A finite group $H$ is said to be \emph{hyperelementary}
  if it can be written as an extension 
  $1 \to C \to H \to P \to 1$, where $C$ is a cyclic group
  and $P$ is a $p$-group for some prime $p$.

  \begin{definition}[Farrell-Hsiang group]
    \label{def:Farrell-Hsiang}
    Let $\calf$ be a family of subgroups of the finitely generated group $G$.
    We call $G$ a \emph{Farrell-Hsiang group} with respect to the family $\calf$
    if the following holds for a fixed word metric $d_G$: 
    
    There exists a natural  number $N$ such that for every natural 
    number $n$ there is a surjective homomorphism $\alpha_n
    \colon G \to F_n$ with $F_n$ a finite group such that the 
    following condition is
    satisfied. For any hyperelementary subgroup $H$ of $F_n$ we set $\ol{H} :=
    \alpha_n^{-1}(H)$ and require that there exists a simplicial complex $E_H$ of
    dimension at most $N$ with a cell preserving simplicial $\ol{H}$-action
    whose stabilizers belong to $\calf$, and an $\ol{H}$-equivariant map $f_H
    \colon G \to E_H$ such that $d_G(g,h) < n$ implies $d_{E_H}^1 (f_H(g),f_H(h)) <
    \frac{1}{n}$ for all $g,h \in G$, where $d^1_{E_H}$ is the $l^1$-metric on $E_H$.
  \end{definition}

  \begin{theorem}[Main Theorem]
    \label{the:FH_implies_FJ}
    Let $G$ be a Farrell-Hsiang group with respect to the 
    family $\calf$ in the sense of Definition~\ref{def:Farrell-Hsiang}.
    Then $G$ satisfies the $K$-theoretic and $L$-theoretic 
    Farrell-Jones Conjecture with additive categories as coefficients with
    respect to the family $\calf$.
  \end{theorem}

   For the precise formulation and discussion of the Farrell-Jones Conjecture with coefficients in additive categories
   we refer to~\cite{Bartels-Lueck(2009coeff)}. 

  \begin{remark}
    Definition~\ref{def:Farrell-Hsiang} can be weakened if one is only 
    interested in the $L$-theoretic Farrell-Jones conjecture.
    In this case it suffices to consider all subgroups $H$ of $F$
    that are either $2$-hyperelementary or $p$-elementary for some 
    prime $p \neq 2$.
    In other words $p$-hyperelementary subgroups that are not
    $p$-elementary can be ignored for all odd primes $p$.
  \end{remark}


\typeout{-------------- Section 2: Categorical preliminaries   -------------}
  
  \section{Categorical preliminaries}
     \label{sec:cat-prelim}
      
     \subsection{Additive $G$-categories with involutions.}
     In this paper we will understand notions like additive category (with
     involution) or additive $G$-category (with involution) 
     always in the strict sense. 
     This means that all our additive categories will come with a
     strictly associative functorial direct sum $(M,N) \mapsto M \oplus N$ 
     and an involution $I$ on an additive category $\calb$ is a 
     contravariant functor $I \colon \calb \to \calb$
     with $I^2 = \id_\calb$. 
     When we talk about an additive $G$-category, the
     (right) $G$-action is understood to be in the strict 
     sense, i.e., for every
     $g \in G$ we have a functor $R_g \colon \calb \to \calb$ of additive
     categories such that $R_h \circ R_g = R_{gh}$ for $g,h \in
     G$.  
     If $\calb$ comes with an
     involution $I_{\calb}$, then we require 
     $I_{\calb} \circ R_g = R_g \circ I_{\calb}$ for all 
     $g \in G$.
 
  \begin{remark}
    Often a more general definition of additive categories with
    involutions is used, where the equality $I^2 = \id_\calb$ 
    is replaced by a natural equivalence $E \colon I^2 \to \id_\calb$.
    One may also consider additive categories with weak $G$-actions.
    We refer to~\cite{Bartels-Lueck(2009coeff)}, where all these notion are explained,
    it is shown how one can replace the weak versions by equivalent strict versions, and
    \--- most important \--- that for a proof of the Farrell-Jones Conjecture it suffices
    to consider the strict versions (see~\cite[Theorem~0.2]{Bartels-Lueck(2009coeff)}). 
    We use the strict versions to simplify some formulas. 
    The only slight disadvantage of this is, that it forces us to replace
    some very natural categories by some slightly less natural 
    categories, see for instance the definition of $\modules_\IZ$ below. 
  \end{remark}

  A functor between additive categories with involutions $(\calb,I)$
  and $(\calb',I')$ is a pair $(F,E)$ where 
  $F \colon \calb \to \calb'$ is an additive  functor,
  and $E \colon F \circ I \to I' \circ F$ is a natural
  equivalence such that
  $I'(E(M)) = E(I(M))$ for all objects $M \in \calb$.
  If $F \circ I = I' \circ F$ and $E = \id$, then
  the functor is said to be strict.
  Most of our functors will be strict, but not all
  of them. 
  Functors between additive categories with involutions induce
  maps in $L$-theory.


  \subsection{The category $\modules_\IZ$ of based
     finitely generated free abelian groups.}
  On the category of finitely generated free abelian groups the 
  involution $T \mapsto T^* := \Hom_{\IZ}(-,\IZ)$ is not
  strict since $T$ is not $(T^*)^*$ on the nose.
  To fix this inconvenience we will consider the following additive 
  category with involution $\modules_\IZ$ instead. 
  The objects of $\modules_\IZ$ are $\IZ^n$, $n=0,1,2,\dots$.
  The set of morphisms
  $\mor_{\modules_\IZ}(\IZ^n,\IZ^m)$ is given by $n \x m$-matrices. 
  Composition is given by the usual matrix multiplication. 
  The direct sum is given by $\IZ^n \oplus \IZ^m = \IZ^{n + m}$. 
  The involution on $\modules_\IZ$ acts as
  the identity on objects and as transposition of matrices on
  morphisms. 
  For an additive category $\cala$ there is a functor  
  \begin{equation*}
    - \ox_\IZ - \colon \modules_\IZ \x \cala \to \cala
  \end{equation*}
  defined by $\IZ^n \ox_\IZ M = \bigoplus_{i=1}^n M$,
  see for example~\cite[Section~6]{Bartels-Lueck(2009borelhyp)}.
  This functor is bilinear on morphisms groups.
  It follows that given
  an object $\IZ^n$ in $\modules_\IZ$, the functor 
  $\IZ^n \ox_\IZ - \colon \cala \to \cala$ is a functor of 
  additive categories, and given an object $M \in \cala$, the
  functor $- \ox_\IZ M \colon \modules_\IZ \to \cala$ is a functor of
  additive categories.
  If $\cala$ comes with an involution, 
  then $\modules_\IZ \x \cala$ inherits the
  obvious product involution and $- \ox_\IZ -$ is compatible with the
  involutions.

  
  \subsection{The category $\modules_{(\IZ,G)}$ of 
      $\IZ G$-modules
      which are  finitely generated free as abelian groups.}
     \label{subsec:mod_(Z,G)}

  Let $G$ be a group. 
  We define the following additive category with involution 
  $\modules_{(\IZ,G)}$. 
  Objects in $\modules_{(\IZ,G)}$ are pairs $(\IZ^n, \rho)$ where  
  $\rho \colon G \to \GL(n,\IZ)$ is a group homomorphism.
  A morphism $f \colon (\IZ^n,\rho) \to (\IZ^m,\eta)$ is a morphism
  $f \colon \IZ^n \to \IZ^m$ in $\modules_{\IZ}$ 
  which is compatible with the homomorphisms $\rho$ and $\eta$, 
  i.e., $\eta(g) \circ f = f \circ \rho(g)$ 
  for all $g \in G$. 
  The direct sum is given by the direct sum in $\modules_{\IZ}$.
  Define an involution $I_{\modules_{(\IZ,G)}}$ on 
  $\modules_{(\IZ,G)}$ as follows. 
  It sends an object $(\IZ^n,\rho)$ to the object $(\IZ^n,\rho^*)$, where
  $\rho^*(g)$ is defined by $I_{\modules_{\IZ}}(\rho(g^{-1}))$. 
  A morphism $f \colon (\IZ^n,\rho) \to (\IZ^m,\eta)$ 
  is sent to the morphism given by 
  $I_{\modules_{\IZ}}(f)$.

  Of course $\modules_{(\IZ,G)}$ is a model for the category of 
  $\IZ G$-modules 
  which are finitely generated free as abelian groups and 
  has the extra feature that the involution is strict. 

  Let $\alpha \colon H \to G$ be a group homomorphism. 
  We obtain a functor
  of additive categories with involution called \emph{restriction}
  \begin{eqnarray*}
  & \res_{\alpha} \colon \modules_{(\IZ,G)} \to \modules_{(\IZ,H)}, &
  \end{eqnarray*}
  which sends an object $(\IZ^n,\rho)$ to the object 
  $(\IZ^n, \rho \circ \alpha)$
  and a morphism $f \colon (\IZ^n,\rho) \to (\IZ^m,\eta)$ to the morphism 
  $f \colon (\IZ^n ,\rho \circ \alpha) \to (\IZ^m, \eta \circ \alpha)$.

  Next we define the induction functor
  for a subgroup $H$ of $G$ of finite index
  \begin{eqnarray*}
    & \ind_H^G \colon \modules_{(\IZ,H)} \to \modules_{(\IZ,G)}. &
  \end{eqnarray*}
  It will depend on a choice of 
  representatives $g_0,\dots,g_{m-1} \in G$ for 
  $G/H = \{g_0\overline{H}, \ldots, g_{m-1}\overline{H}\}$.
  This choice will not matter in the sequel, since for two
  such choices we obtain a unique natural equivalence of the corresponding
  functors of additive categories with involution. 
  Consider an object
  $(\IZ^n,\rho)$ in $\modules_{(\IZ,H)}$. 
  The image under $\ind_H^G$ is the
  object $(\IZ^{m \cdot n}, \eta)$, where 
  $\eta(g) \in \GL(m \cdot n,\IZ)$ for $g \in G$ is the morphism in
  $\modules_\IZ$ given by the matrix whose entry at 
  $(kn+i,k'n+i')$ is $0$ if
  $g g_{k'} H \not= g_{k} H,$ and is 
  $\rho \bigl(g_{k}^{-1}g g_{k'} \bigr)_{i,i'}$
  if $g g_{k'} H = g_{k}H$. 
  Here $0 \leq k,k' \leq {m-1}$, $1 \leq i,i' \leq n$.

  \begin{remark} \label{rem:induction}
    Our above definition of $\ind_H^G$ may appear unnatural.
    But the only reason for this is our choice of the category
    $\modules_{(\IZ,H)}$; 
    it really is the usual definition of induction: 

    Let $(\IZ^n,\rho)$ be an object of $\modules_{(\IZ,H)}$.
    Then $\IZ^n$ becomes an $\IZ[H]$-module via $\rho$.
    We have the following isomorphism of $\IZ$-modules
    \begin{equation*}
      \IZ[G] \ox_{\IZ[H]} \IZ^n 
        \cong \bigoplus_{j=0}^{m-1} \IZ[g_jH] \ox_{\IZ[H]} \IZ^n
        \cong \bigoplus_{j=0}^{m-1} \IZ^n
        \cong \IZ^{m \cdot n}
    \end{equation*}
    and the above formula for $\eta$
    describes how the action of $G$ on $\IZ[G] \ox_{\IZ[H]} \IZ^n$
    conjugates to an action on $\IZ^{m \cdot n}$ under the above
    isomorphism.
  \end{remark}


\typeout{-------------- Section 3: The obstruction category   -------------}

  \section{The obstruction 
         category $\calo^G(E,Z,d;\cala)$}
    \label{sec:calo}
  
  Let $E$ be a $G$-space and $(Z,d)$ be a quasi-metric space
  with a free, proper and isometric $G$-action.
  In this section we will review the 
  the additive category $\calo^G(E,Z,d;\cala)$ that was 
  originally defined in~\cite[Section~3]{Bartels-Lueck-Reich(2008hyper)},
  see also~\cite[Section~4]{Bartels-Lueck(2009borelhyp)}. 
  If $\cala$ is an additive category with involution, 
  then $\calo^G(E,Z,d;\cala)$ is an additive category with
  involution.
  
  \subsection{Objects.}
  Objects in $\calo^G(E,Z,d;\cala)$ are given by
  sequences $M = (M_{y})_{y \in Z \x E \x \einsu}$ of 
  objects from $\cala$ subject to the following conditions.
  \begin{enumerate}
  \item \emph{$G$-compact support over $Z \x E$.}
        There is a compact subset $K$ of $Z \x E$ such that
        $M_{z,e,t} = 0$ whenever $(z,e) \not\in G \cdot K$.
  \item \emph{Locally finiteness.}
        For all $y \in Z \x E \x \einsu$ 
        there exists an open neighborhood
        $U$ such that
        $\{ y \in U \mid M_{y} \neq 0 \}$
        is finite.
  \item \emph{$G$-equivariance.} 
        For all $y \in Z \x E \x \einsu$ and $g \in G$
        we have $M_{gy} = g(M_{y})$.
        Here $gy = (gz,ge,t)$ for $y = (z,e,t)$.   
  \end{enumerate}
  The involution $I_\calo$ on $\calo^G(E,Z,d_G;\cala)$ acts on objects
  point-wise, i.e., we have
  $(I_\calo(M))_{z,e,t} = I_\cala(M_{z,e,t})$. 

  \subsection{Morphisms.}
  Let $M = (M_{y})_{y \in Z \x E \x \einsu}$, $N = (M_y)_{y \in Z \x E \x \einsu}$ 
  be objects from $\calo^G(E,Z,d)$. A morphism $\psi \colon M \to N$
  in $\calo^G(E,Z,d)$ is given by a sequence 
  $\psi = (\psi_{y,y'})_{y,y' \in Z \x E \x \einsu}$ of morphisms
  $\psi_{y,y'} \colon M_{y'} \to N_y$ in $\cala$ subject to
  the following conditions.
  \begin{enumerate}
  \item \emph{Row and column finiteness.}
        For all $y \in Z \x E \x \einsu$ the set
        $\{ y' \mid \psi_{y,y'} \neq 0  \; \text{or} \; 
             \psi_{y',y} \neq 0 \}$ is finite.
  \item \emph{Metric control over $Z$.}
        There is $R > 0$ (depending on $\psi$) such
        that $\psi_{y,y'} = 0$ whenever
        $y = (z,e,t)$, $y' = (z',e',t')$ with
        $d(z,z') > R$.
  \item \emph{Metric control over $\einsu$.}
        There is $A > 0$ (depending on $\psi$) such
        that $\psi_{y,y'} = 0$ whenever
        $y = (z,e,t)$, $y' = (z',e',t')$ with
        $| t - t'| > A$.
  \item \emph{$G$-continuous control over $E \x \einsu$.}
        Let $e_0 \in E$, $V$ be an $G_{e_0}$-invariant 
        neighborhood of $e_0$ and $b > 0$. 
        (Here $G_{e_0} = \{ g \mid ge_0=e_0 \}$.)
        Then we require the existence of $B > 0$ and 
        a $G_{e_0}$-invariant neighborhood $U$ of $e_0$ such that
        $\psi_{y,y'} = \psi_{y',y} = 0$ whenever
        $y = (z,e,t)$, $y' = (z',e',t')$ with
        $(e,t) \in U \x (B,\infty)$ and
        $(e',t') \not\in V \x (b,\infty)$.
  \item \emph{$G$-equivariance.}
        For all $y,y' \in Z \x E \x \einsu$ and $g \in G$
        we have $\psi_{gy,gy'} = g(\psi_{y,y'})$.
  \end{enumerate}
  For the constructions in this paper the second condition will
  be the most important condition and we will say that $\psi$ is
  \emph{$R$-controlled} if it is satisfied for a given $R > 0$.
  Addition and composition of morphisms is defined as for matrices:
  $(\psi + \psi')_{y,y'} = \psi_y + \psi'_{y'}$ and
  $(\psi \circ \psi')_{y,y''} = \sum_{y'} \psi_{y,y'} \circ \psi_{y',y''}$.
  The involution is on morphisms defined by the formula
  $(I_\calo(\psi))_{y,y'} = I_\cala(\psi_{y',y})$.
  
  We will often drop $\cala$ from the notation and write
  $\calo^G(E,Z,d)$ instead of \linebreak $\calo^G(E,Z,d;\cala)$.

  \subsection{Functoriality.}
  In this paper we will only need the functoriality of 
  $\calo^G(E,Z,d;\cala)$ in the $Z$-variable.
  Let  $(Z,d)$ and $(Z',d')$ be quasi-metric spaces
  with free, proper and isometric $G$-actions.
  Let $f \colon Z \to Z'$ be $G$-equivariant continuous map
  such that for any $r > 0$ there is $R > 0$ such that
  $d'(f(z_0),f(z_1)) < R$ whenever $d(z_0,z_1) < r$.
  Then $f$ induces a functor $f_* \colon \calo^G(E,Z,d) \to \calo^G(E,Z',d')$
  which is given by $(f_*(M))_{e,z,t} = \bigoplus_{z' \in f^{-1}(z)}
  M_{e,z',t}$.  (The condition ensures that metric control over $Z$ is
  turned into metric control over $Z'$; the $G$-compact support 
  condition for objects ensures that the sum in the
  definition of $f_*$ is finite.)  
  Strictly speaking $f_*$ is only defined up to natural
  equivalence because the direct sum may only be defined up to canonical
  isomorphism.
  (Our assumptions on $\cala$ only imply that 
  sums over \emph{ordered} finite index set
  are canonically defined.)
   

  \subsection{$\calo^G(E,G,d)$ as the obstruction to
       the Farrell-Jones conjecture.}

  The following result is a consequence 
  of~\cite[Theorem~5.2]{Bartels-Lueck(2009borelhyp)}. 

  \begin{theorem}
     \label{thm:calo-and-FJ}
    Let $G$ be a finitely generated group, $d_G$ a word metric on $G$ 
    and $\calf$ be a family of subgroups. 
    \begin{enumerate}
    \item Assume that $K_*(\calo^G(E_\calf G, G, d_G))$ is trivial
       in all degrees.
       Then the $K$-theory assembly map
       $H^G_*(E_\calf G; \bfK_\cala) \to K_*(\int_G \cala)$
       is an isomorphism.
    \item Assume that $L_*(\calo^G(E_\calf G, G, d_G))$ is trivial
       in all degrees.
       Then the $L$-theory assembly map
       $H^G_*(E_\calf G; \bfL_\cala) \to L_*(\int_G \cala)$
       is an isomorphism.  
     \end{enumerate}
  \end{theorem}


  \subsection{The controlled product category.}

  Let $(Z_n,d_n)$ be a sequence of quasi-metric spaces
  with free, proper and isometric $G$-actions.
  Consider the product category $\prod_{n \in \IN} \calo^G(E,Z_n,d_n)$.
  A morphism $\varphi = (\varphi_n)_{n \in \IN}$ is said to be $R$-controlled
  for $R > 0$ is $\varphi_n$ is $R$-controlled for all $n$.
  We define $\calo^G(E,(Z_n,d_n)_{n \in \IN})$ as the category whose
  objects are objects from the product category and whose morphisms
  are morphisms from the product category that are $R$-controlled for
  some $R$. 
  There is for any $k$ a canonical projection functor
  $\calo^G(E,(Z_n,d_n)_{n \in \IN}) \to \calo^G(E,Z_k,d_k)$. 


\typeout{----- Section 4: The Core of the proof of the main Theorem --------}  

  \section{The Core of the proof of the main 
                         Theorem~\ref{the:FH_implies_FJ}}
    \label{sec:core}

  Let $G$ be a Farrell-Hsiang group with respect to $\calf$.
  Let $N$ be the number appearing in Definition~\ref{def:Farrell-Hsiang}.
  For $n \in \IN$ there is then $\alpha_n \colon G \to F_n$, a 
  surjective group homomorphism onto a finite group $F_n$,
  such that the following holds:
  For any hyperelementary subgroup $H$ of $F_n$ and 
  $\ol{H} := {\alpha_n}^{-1}(H)$ there is a
  simplicial complex $E_H$ of dimension at most $N$ with a cell
  preserving simplicial $\ol{H}$-action whose 
  stabilizers belong to $\calf$, and an $\ol{H}$-equivariant map 
  $f_H \colon G \to E_H$ such that
  $d_G(g,h) < n$ implies $d_{E_H}^{1} (f_H(g),f_H(h)) < \frac{1}{n}$ 
  for all $g,h \in G$,  where $d^{1}_{E_H}$ is the $l^1$-metric on $E$.
  Here we write $\ol{H}$ for ${\alpha_n}^{-1}(H)$ and we will
  use this convention throughout the remainder of this paper.
  We denote by $\calh_n$ the family of hyperelementary subgroups of $F_n$.
  We set $X_n := G \x \coprod_{H \in \calh_n} \ind_{\ol{H}}^G E_H$
  and $S_n := G \x  \coprod_{H \in \calh_n} G/\ol{H}$. 
  We equip $X_n$ and $S_n$ with diagonal $G$-action.
  We will use the quasi-metrics $d_{X_n}$ on $X_n$ 
  and $d_{S_n}$ on $S_n$ defined by
  \begin{eqnarray*}
    d_{X_n} ( (g,x), (h,y) ) & := & d_G(g,h) + 
                      n \cdot d^1_{\ind^G_{\ol{H}}E_H}(x,y), \\
    d_{S_n} ( (g,a\ol{H}), (h,b\ol{K}) ) & := & 
                 \begin{cases} 
                    d_G(g, h) & \text{if} \; K=H \; \text{and} \; 
                        a \ol{H}=b \ol{K}, \\
                    \infty & \text{otherwise}. 
                 \end{cases}
  \end{eqnarray*}
  Here $g, h, a, b \in G$, $x, y \in X_n$, $H, K \in \calh_n$
  and $d^1_{\ind^G_{\ol{H}} E_H}$ is the $l^1$-metric on $\ind^G_{\ol{H}} E_H$.
  Abbreviate $E := E_\calf G$.
  The proof of Theorem~\ref{the:FH_implies_FJ} is organized
  around the following diagram of additive categories and functors.
  \begin{equation}
    \label{eq:main-diagram}
    \xymatrix{& & 
       \bigoplus_{n \in \IN} \calo^G(E, X_n, d_{X_n}) 
       \ar[d]^{I}
       \\
       \calo^G(E, (S_n, d_{S_n})_{n \in \IN})
       \ar[d]^{P_k} \ar[rr]^{F} 
       & &
       \calo^G(E, (X_n, d_{X_n})_{n \in \IN}) 
       \ar[d]^{Q_k}
       \\
       \calo^G(E, G, d_G)
       \ar[rr]^{\id}
       & &
       \calo^G(E, G, d_G)
    } 
  \end{equation}
  Explanations follow.
  The functors ${P}_k$ and ${Q}_k$ are defined as compositions
  \begin{eqnarray*}
    \calo^G(E, (S_n, d_{S_n})_{n \in \IN}) \to & 
    \calo^G(E, S_k, d_{S_k}) \to &
    \calo^G(E, G, d_G) 
    \\
    \calo^G(E, (X_n, d_{X_n})_{n \in \IN}) \to & 
    \calo^G(E, X_k, d_{X_k}) \to &
    \calo^G(E, G, d_G) 
  \end{eqnarray*}
  where in both cases the first functor is the 
  projection on the k-th factor,
  and the second functor is induced by the canonical projection 
  $p_k \colon S_k = G \x \coprod_{H \in \calh_k} G/\ol{H} \to G$ and
  $q_k \colon X_k = G \x \coprod_{H \in \calh_k} \ind_{\ol{H}}^G E_H \to G$
  respectively.  
  The functor $I$ is the canonical inclusion.
  The functor  $F$ will be constructed in Proposition~\ref{prop:F}.
  We have the following facts.
  \begin{enumerate}
  \item For all $a \in K_*(\calo^G(E,G,d_G))$ and 
     $b \in L_*(\calo^G(E,G,d_G))$ 
     there are $\hat a \in K_*(\calo^G(E,(S_n,d_{S_n})_{n \in \IN}))$ 
     and $\hat b \in L_*(\calo^G(E,(S_n,d_{S_n})_{n \in \IN}))$ 
     such that for all $k$ we have $\big( K_n(P_k) \big)(\hat a) = a$
     and $\big( L_*(Q_k) \big)(\hat b) = b$. 
     This will be proved in Theorem~\ref{thm:onto}.
  \item For all $k$ we have $Q_k \circ F = P_k$,
     see Proposition~\ref{prop:F}.
  \item The functor $I$ induces an isomorphism in $K$- and $L$-theory.
     For $K$-theory this follows 
     from~\cite[Theorem~7.2]{Bartels-Lueck-Reich(2008hyper)}.
     This result only depends on the properties of $K$-theory
     listed in~\cite[Theorem~5.1]{Bartels-Lueck(2009borelhyp)}.
     Since these properties are also enjoyed by $L$-theory,
     $I$ induces an isomorphism in $L$-theory as well.
  \end{enumerate}

  \begin{proof}[Proof of Theorem~\ref{the:FH_implies_FJ}]
     Because of Theorem~\ref{thm:calo-and-FJ} it suffices 
     to show that the $K$- and $L$-theory of $\calo^G(E,G,d_G)$
     is trivial.
     Let $a \in K_n(\calo^G(E,G,d_G))$ and
     $b \in L_n(\calo^G(E,G,d_G))$.
     By the first fact there are 
     $\hat a \in K_n(\calo^G(E,(S_n,d_{S_n})_{n \in \IN}))$ 
     and $\hat b \in L_n(\calo^G(E,(S_n,d_{S_n})_{n \in \IN}))$ 
     such that for all $k$ we have $\big( K_n(P_k) \big)(\hat a) = a$
     and $\big( L_n(P_k) \big)(\hat b) = b$.
     It is a consequence of the third fact that for sufficient large
     $k$ we have 
     $\big( K_n(Q_k \circ F) \big)(\hat a) = 0$
     and $\big( L^{-\infty}_n(Q_k \circ F) \big)(\hat b) = 0$.
     Using the second fact we conclude
     $a = \big( K_n(P_k) \big)(\hat a) = 
        \big( K_n(Q_k) \circ F \big)(\hat a) = 0$ and
     $b = \big( L_n(P_k) \big)(\hat b) =
      \big( L_n(Q_k \circ F) \big)(\hat b) = 0$.
     (Compare~\cite[p.45, Proof of Theorem~1.1]
         {Bartels-Lueck-Reich(2008hyper)}.) 
  \end{proof}


\typeout{----- Section 5: Abstract transfers for additive categories --------}

  \section{Abstract transfers for additive categories}
    \label{sec:abstr-transfer}


  \subsection{Swan group and Dress' equivariant Witt group}

  We have introduced the additive category with involutions $\modules_{(\IZ,G)}$
  in Section~\ref{sec:cat-prelim}. Recall that it is equivalent to the category
  of $\IZ G$-modules which are finitely generated free as $\IZ$-modules.
  We will use the exact structure on $\modules_{(\IZ,G)}$ where a sequence is
  called exact if it is exact as a sequence of $\IZ[G]$-modules (or equivalently
  as a sequence of abelian groups).  
  Notice that with this exact structure not all exact
  sequences are split exact over $\IZ G$.  
  The \emph{Swan group} and \emph{Dress' equivariant Witt group} 
  are defined with respect to this exact structure as corresponding Grothendieck 
  or Witt groups
  \begin{equation*}
     \Sw(\IZ,G) := G_0(\modules_{(\IZ,G)}) 
     \quad \quad \text{and} \quad
     \GW(\IZ,G) := W(\modules_{(\IZ,G)}), 
  \end{equation*}
  see~\cite{Swan(1960a), Dress(1975), Lueck-Ranicki(1988)}.
  Both of these become rings via the tensor product over $\IZ$,
  equipped with the diagonal $G$-action, but we will not need 
  this ring structure and ignore it in this paper.
  By $1_{\Sw} \in \Sw(\IZ,G)$ we will denote the class of
  $\IZ$  with the trivial  $G$-action 
  and analogously for $1_{\GW} \in \GW(\IZ,G)$. 
  (These are of course the units for the ring structures.)
  
  For a group homomorphism $\alpha \colon H \to G$ 
  there are restriction maps
  \begin{eqnarray*}
       \res_\alpha \colon \Sw(\IZ, G) \to \Sw(\IZ,H) 
       \\
       \res_\alpha \colon \GW(\IZ, G) \to \GW(\IZ,H)
  \end{eqnarray*}
  coming from the restriction functor 
  $\res_\alpha \colon \modules_{(\IZ,G)} \to \modules_{(\IZ,H)}$.
  Clearly, we have $\res_\alpha(1_{\Sw}) = 1_{\Sw}$ and
  $\res_\alpha(1_{\GW}) = 1_{\GW}$.
   
  For a subgroup $H \subseteq G$ of finite index there are induction homomorphisms
  \begin{eqnarray*}
       \ind_H^G \colon \Sw(\IZ,H) \to \Sw(\IZ,G)  
       \\
       \ind_H^G \colon \GW(\IZ, H) \to \GW(\IZ,G)
  \end{eqnarray*}
  coming from the induction functor 
  $\ind_H^G \colon \modules_{(\IZ,H)} \to \modules_{(\IZ,G)}$.

  Actually, both $\Sw(\IZ,-)$ and $\GW(\IZ,-)$ are Green functors.
  Later on we will make crucial use of the following results
  due to Swan and Dress.

  \begin{theorem}[Swan~\cite{Swan(1960a)};Dress\cite{Dress(1975)}]
    \label{thm:Swan_and_Dress}
    Let $F$ be a finite group. 
    Let $\calh$  be the family of hyperelementary subgroups
    of $F$.
    \begin{enumerate}
    \item \label{thm:Swan_and_Dress:Swan}
       There are $\tau_H \in \Sw(\IZ,H)$, $H \in \calh$
       such that 
       \begin{equation*}
          1_{\Sw} = \sum_{H \in \calh} \ind^F_H(\tau_H) \in \Sw(\IZ,F).
       \end{equation*}
    \item \label{thm:Swan_and_Dress:Dress} 
       There are $\sigma_H \in \GW(\IZ,H)$, $H \in \calh$
       such that 
       \begin{equation*}
          1_{\GW} = \sum_{H \in \calh} \ind^F_H(\sigma_H) \in \GW(\IZ,F).
       \end{equation*}
    \end{enumerate}
  \end{theorem}

  \begin{remark} \label{rem:not-all-hyperelem}
    In Theorem~\ref{thm:Swan_and_Dress}~\ref{thm:Swan_and_Dress:Dress}
    the family $\calh$ can be replaced by the family of subgroups $H$ of $F$
    that are either $2$-elementary or $p$-hyperelementary for some 
    prime $p \neq 2$.
  \end{remark}


  \subsection{Action of $\Sw(\IZ,G)$ in $K$-theory.}
  Let $R$ be a ring and $G$ be a group.  Denote by $\modules_{R[G]}$ the
  category of finitely generated projective $R[G]$-modules.  
  The tensor product over $\IZ$, equipped with
  the diagonal $G$-action, $(T,M) \mapsto T \ox_\IZ^\Delta M$ defines a bilinear
  functor
  \begin{equation*}
     - \ox_\IZ^\Delta - \colon 
       \modules_{(\IZ,G)} \x \modules_{R[G]} \to \modules_{R[G]}.
  \end{equation*}
  In particular, 
  we obtain a functor 
  $T \ox_\IZ^\Delta - \colon \modules_{R[G]} \to \modules_{R[G]}$
  for every  module $T \in \modules_{(\IZ,G)}$.
  Applying $K$-theory we obtain a endomorphism 
  $K_*(T \ox_\IZ^\Delta -)$ of $K_*(R[G])$,
  This endomorphism depends only on the class of $T$ in
  $\Sw(\IZ,G)$ and yields  a pairing
  \begin{equation*}
    \mu \colon \Sw(\IZ,G) \ox K_*(R[G]) \to K_*(R[G])
  \end{equation*}
  such that $\mu( [T] \ox a) = K_*( T \ox_\IZ^\Delta -)(a)$ 
  for all $a \in K_*(R[G])$ (see~\cite[Corollary~1 on page~106]{Quillen(1973)}).
  This has  a generalization as follows.
  For an additive category $\calb$ a functor   
  \begin{equation*}
      F \colon \modules_{(\IZ,G)} \x \calb \to \calb   
  \end{equation*}
  is said to be \emph{exact} if $F$ is bilinear and for any 
  short exact sequence
  (which is \emph{not} necessarily split exact) $0 \to S_0 \xrightarrow{i} S_1
  \xrightarrow{p} S_2 \to 0$ in $\modules_{(\IZ,G)}$ and any object $B$ in
  $\calb$ the induced sequence 
  $0 \to F(S_0,B) \xrightarrow{F(i,\id_B)} F(S_1,B)
  \xrightarrow{F(p,\id_B)} F(S_2,B) \to 0$ in $\modules_{(\IZ,G)}$ is exact in
  $\calb$. Recall that a sequence $0 \to B_0 \xrightarrow{j} B_1 \xrightarrow{q}
  B_2 \to 0$ in an additive category $\calb$ is called exact if it is split 
  exact, i.e, $q \circ j = 0$  and there exists a morphism 
  $s \colon B_2 \to B_0$ such that
  $q \circ s = \id_{B_2}$ and $j \oplus s \colon B_0 \oplus B_2 \to B_1$ is an
  isomorphism.

  \begin{proposition}
    \label{prop:swan-acts-on-B}
    Given an exact functor
    $F \colon \modules_{(\IZ,G)} \x \calb \to \calb$,
    there is a bilinear pairing  
    \begin{equation*}
      \mu_F \colon \Sw(\IZ,G) \ox K_*(\calb) \to K_*(\calb)
    \end{equation*}
    such that $\mu_F( [T] \ox a) = K_*(F(T,-))(a)$ for all 
    $a \in K_*(\calb)$.
  \end{proposition}

  \subsection{Action of $\GW(\IZ,G)$ in $L$-theory}
  Let $\calb$ be an additive category with a strict involution 
  $I_\calb$ and 
  \begin{equation*}
     F \colon \modules_{(R,G)} \x \calb \to \calb   
  \end{equation*}
  be an exact functor which is compatible with the
  involutions, i.e., $I_\calb(F(-,-)) = F(-^*,I_\calb(-))$.
  Then for a module $T \in \modules_{(G,\IZ)}$ the 
  linear functor $F(T,-) \colon \calb \to \calb$ does a priori
  not induce a map in $L$-theory because no canonical
  isomorphism $I_\calb(F(T,M)) \to F(T, I_\calb(M))$ is provided.
  To fix this, we pick an isomorphism
  $\varphi \colon T \to T^*$ in $\modules_{(\IZ,G)}$
  such that $\varphi^* = \varphi$, so $(T,\varphi)$
  is a symmetric form in $\modules_{(\IZ,G)}$.
  Then 
  \begin{equation*}
    F(\varphi, \id_{I_\calb(-)}) \colon F(T,I_\calb(-)) 
         \to F(T^*,I_\calb(-)) = I_\calb( F (T,-))
  \end{equation*}
  is a natural isomorphism and 
  $F((T,\varphi),-) := (F(T,-), F(\varphi, \id_{I_\calb(-)}))
     \colon \calb \to \calb$ is a functor of additive categories
  with involutions.
  There is the following analog of Proposition~\ref{prop:swan-acts-on-B}.

  \begin{proposition}
    \label{prop:GW-acts-on-B}
    Given an exact functor 
    $F \colon \modules_{(G,\IZ)} \x \calb \to \calb$
    that is compatible with involutions, there is a bilinear pairing
    \begin{equation*}
    \mu_F \colon \GW(\IZ,G) \ox L_*(\calb) \to L_*(\calb)
  \end{equation*}
  such that $\mu_F ( [T,\varphi] \ox b ) = L_*(F((T,\varphi),-) )(b)$
  for all $b \in L_*(\calb)$ and all symmetric forms $(T,\varphi)$
  over $\modules_{(\IZ,G)}$.
  \end{proposition}
   
  \begin{proof}
    If $\calb$ is the category of finitely generated free 
    $R[G]$-modules and $F$ is the diagonal tensor product,
    then this is worked out in detail in~\cite{Dress(1975)}
    and~\cite{Lueck-Ranicki(1988)}.
    The case of general $F$ and $\calb$ is not more complicated.
      \end{proof}
  

\typeout{-------------------- Section 6: The transfer ----------------------}  

  \section{The transfer}
    \label{sec:transfer}


  \subsection{Transfer functors}

  Let $G$ be a group with a metric $d_G$ and $E$ be a $G$-space.
  We define a functor
  \begin{equation*}
    \tr \colon \modules_{(\IZ,G)} \x \calo^G(E,G,d_G) \to \calo^G(E,G,d_G)
  \end{equation*}
  as follows.
  Recall that we have a  tensor product
  functor $\modules_\IZ \x \cala \to \cala$ which is compatible 
  with the involution on $\modules_\IZ$ and $\cala$, see 
  Section~\ref{sec:cat-prelim}.
  For objects $T = (\IZ^n,\rho) \in \modules_{(\IZ,G)}$ and 
  $M = (M_z)_{z \in G \x E \x \einsu} \in \calo^G(E,G,d_G)$ 
  we define $\tr(T,M) \in \calo^G(E,G,d_G)$ by setting
  \begin{equation*}
    \big( \tr(T,M) \big)_z := \IZ^n \ox_\IZ M_z
  \end{equation*}
  for $z \in  G \x E \x \einsu$.
  For morphisms $f \in \modules_{(\IZ,G)}$ and  
  $\psi = (\psi_{z,z'})_{z,z' \in G \x E \x \einsu} \in \calo^G(E,G,d_G)$ 
  we define $\tr(f,\psi)$ by setting
  \begin{equation*}
    \big( \tr(f,\psi) \big)_{z,z'} := (f \circ \rho(g^{-1}g')) \ox_\IZ \psi_{z,z'}
  \end{equation*}
  for $z = (g,e,t), z'= (g',e',t') \in  G \x E \x \einsu$.
  
  \begin{lemma}
    \label{lem:tr-is-exact}
    The functor $\tr$ is exact. It is compatible with involutions
    if $\cala$ comes with a (strict) involution.
  \end{lemma}

  \begin{proof} 
    The compatibility with involutions follows from the same 
    compatibility for $\ox_\IZ$. 

    Consider an exact sequence
    $0 \to S_0 \xrightarrow{i} S_1 \xrightarrow{q} S_2 \to 0$ in 
    $\modules_{(\IZ,G)}$.
    We have to show that for any object $M$ in $\calo^G(E,G,d_G)$
    that the composite $\tr(q,\id_M) \circ \tr(i,\id_M)$ 
    is trivial, $\tr(q,\id_M) \colon \tr(S_1,M) \to \tr(S_2,M)$
    is split surjective,  and that the direct sum of the splitting and the map
    $\tr(i,\id_M)$ yields an isomorphism 
    $\tr(S_0,M) \oplus \tr(S_2,M) \xrightarrow{\cong} \tr(S_1,M))$.  
    We only construct the splitting of  $\tr(q,\id_M)$.
    Let $s \colon S \hookrightarrow T$ be a section for $q$ as a map of
    $\IZ$-modules. Then a section $\hat s$ for $\tr(q,\id_M)$ is defined by setting
    \begin{equation*}
        (\hat s)_{z,z'} := \begin{cases}
                             s \ox \id_{M_z} & \text{if} \; z=z' \\
                             0            & \text{otherwise.}
                          \end{cases}
    \end{equation*}
  \end{proof}
  
  \begin{remark} To illustrate the proof above consider an
  epimorphism $p \colon M \to N$ of $\IZ G$-modules which are finitely generated 
  free as abelian groups and the induced map of $\IZ G$-modules 
  (with respect to the diagonal action) $p \otimes_{\IZ} \id_{\IZ G} \colon 
  M \otimes_{\IZ} \IZ G \to N \otimes_{\IZ} \IZ G$. We want to construct
  a $\IZ G$-splitting. Choose any map of $\IZ$-modules $s \colon N \to M$ with
  $p \circ s = \id_N$. It exists since we do \emph{not} require that $s$ 
  is compatible with the $G$-action. Then a $\IZ G$-splitting of 
  $p \otimes_{\IZ} \id_{\IZ G}$ is given by the $\IZ G$-map
  $N \otimes_{\IZ} \IZ G \to M \otimes_{\IZ} \IZ G$ sending 
  $n \otimes g$ to $gs(g^{-1}n) \otimes g$.
  \end{remark} 
   
  We will need a variant of $\tr$ that combines it with
  an induction map.
  This will yield additional control in the target category
  which is crucial for our argument.
  Let $\alpha \colon G \to F$ be a surjective group homomorphism,
  $H$ be subgroup of finite index in $F$. Put $\overline{H} = \alpha^{-1}(H)$. 
  We have defined induction and restriction
  in Section~\ref{sec:cat-prelim}.  Consider the functor
  \begin{equation*}
    \tr_\alpha := \tr( \res_\alpha \circ \ind_H^F ( - ),-) \colon
     \modules_{(\IZ,H)} \x \calo^G(E,G,d_G) \to \calo^G(E,G,d_G). 
  \end{equation*}
  Define a quasi-metric $d_{G,H}$ on $G \x G/\ol{H}$ by 
  \begin{equation*}
       d_{G,\ol{H}}((g, a \ol{H}),(h, b \ol{H})) :=
          \begin{cases}
             d_G(g,h) & \text{if} \; a \ol{H} = b \ol{H}, \\
             \infty & \text{otherwise}.
          \end{cases}
  \end{equation*}
  The projection $p_H \colon G \x G/\ol{H} \to G$ induces 
  a functor 
  $P_H \colon \calo^G(E, G \x G/\ol{H}, d_{G,\ol{H}}) \to
     \calo(E,G,d_G)$
  and we will see that we can lift $\tr_{\alpha}$ against $P_H$.
  Define a functor
  \begin{equation*}
     \wt{\tr}_{\alpha} \colon \modules_{(\IZ,H)} \x \calo^G(E,G,d_G) \to 
          \calo^G(E,G \x G/\ol{H},d_{G,\ol{H}})   
  \end{equation*}
  as follows.
  For objects $T = (\IZ^n,\rho)\in \modules_{(\IZ,H)}$ and 
  $M = (M_z)_{z \in G \x E \x \einsu} \in \calo^G(E,G,d_G)$ 
  we define $\wt{\tr}_\alpha(T,M)$ by setting
  \begin{equation*}
    \big( \wt{\tr}_{\alpha}(T,M) \big)_y := \IZ^n \ox_\IZ M_z 
  \end{equation*}
  for $y = (g,a\overline{H},e,t) \in G \x G/\ol{H} \x E \x \einsu$
  and $z := (g,e,t)$.
  \ingreen{In order to write out $\wt{\tr}_\alpha$ for morphisms we need to choose 
  representatives $g_0,\dots,g_{m-1} \in G$ for 
  $G/\overline{H} = \{g_0\overline{H}, \ldots, g_{m-1}\overline{H}\}$.}
  For morphisms $f \in \modules_{(\IZ,H)}$
  and $\psi = (\psi_{z,z'})_{z,z' \in G \x E \x \einsu} \in \calo^G(E,G,d_G)$ 
  we define $\wt{\tr}_\alpha(f,\psi)$ by setting 
  \begin{equation*}
    \big( \wt{\tr}_{\alpha}(f,\psi) \big)_{y,y'} := 
            \begin{cases}
             f \circ \rho(\alpha(\ingreen{{g_k}^{-1}g^{-1}g'g_{k'}})) 
                     \ox_\IZ \psi_{z,z'} &
             \text{if} \; \ingreen{gg_k \overline{H} = g'g_{k'} \overline{H}}, \\
             0 & \text{otherwise}.
            \end{cases}
  \end{equation*}
  for \ingreen{$y = (g,gg_k\overline{H},e,t), 
      y' = (g',g'g_{k'}\overline{H},e',t') \in G \x G/\ol{H} \x E \x \einsu$}
  and $z := (g,e,t)$, $z' := (g',e',t')$.
  (The extra $G/\ol{H}$-factor incorporates
  the induction from $H$ to $F$; the appearance of $\alpha$ incorporates
  the restriction along $\alpha$.)

  The following Lemma is a simple exercise in the definitions
  of $\tr_\alpha$ and $\wt{\tr}_\alpha$.

  \begin{lemma}
    \label{lem:tilde-tr}
    \
    \begin{enumerate}
    \item \label{lem:tilde-tr:equivalent} 
       $P_H \circ \wt{\tr}_{\alpha}$ and $\tr_{\alpha}$ are 
       equivalent functors.
    \item \label{lem:tilde-tr:control}
       If $\psi$ is an $R$-controlled
       morphism in $\calo^G(E,G,d_G)$ and $f \in \modules_{(\IZ,H)}$ is any 
       morphism, then $\wt{\tr}_\alpha(f,\psi)$ is $R$-controlled
      in $\calo^G(E,G \x G/\overline{H},d_{G,\overline{H}})$.
    \end{enumerate}
  \end{lemma}

  \begin{proof}~\ref{lem:tilde-tr:equivalent}
    \ingreen{
    To check this we unravel the definitions of 
    $\tr_\alpha$ and $\wt{\tr}_\alpha$ a bit. 
    For $T = (\IZ^n,\rho)$ we have 
    \[     (\res_\alpha \circ \ind_H^F)(\IZ^n,\rho) =
                (\IZ^{nm}, \eta \circ \alpha) 
          = (\bigoplus_{k=0}^{m-1} \IZ^n, \eta \circ \alpha) 
    \]
    where $\eta$ is as defined in the paragraph before 
    Remark~\ref{rem:induction}. 
    It will be helpful to name each of the $m$ copies of $\IZ^n$,
    by $T_0, \dots, T_{m-1}$. 
    Then $\IZ^{nm} = \bigoplus_{k=0}^{m-1} T_k$. 
    Let $z = (g,e,t)  \in G \x E \x \einsu$.
    For $y = (g,g g_k \ol{H},e,t)$ we have 
    $\big( \wt{\tr}_\alpha (T,M) \big)_y 
              = T_k \ox_\IZ M_z$ (as $T_k = \IZ^n$).
    Therefore
    $\big( p_H \circ \wt{\tr}_\alpha (T,M) \big)_z 
        \cong  \bigoplus_{k=0}^{m-1} T_k \ox_\IZ M_z
        = \big( \tr_\alpha(T,M) \big)_z$.  
    In particular, we have a canonical isomorphism  
    $\tau_{T,M} \colon p_H \circ \wt{\tr}_\alpha (T,M) \cong \tr_\alpha(T,M)$.
   }

    \ingreen{
    We have to check that $\tau$ is natural with respect to morphisms
    $(f,\psi)$.
    Inspection of the definition of $\eta$ shows that 
    for $\gamma \in G$ the $(k,k')$-block in $\eta \circ \alpha(\gamma)$
    with respect to $\IZ^{nm} = \bigoplus_{k=0}^{m-1} T_k$ is given by
    \begin{equation*}
           \big( \eta \circ \alpha (\gamma) \big)_{k,k'} = 
             \begin{cases}
                 \rho(\alpha (g^{-1}_{k} \gamma g_{k'})) 
                       & \text{if} \; \gamma g_{k'} 
                           \overline{H} = g_{k} \overline{H}, \\
                0 & \text{otherwise.}
             \end{cases}
    \end{equation*}
    By definition 
    \begin{equation*}
     \big( \tr_\alpha (f,\psi) \big)_{z,z'} = 
         (f \circ \eta(\alpha(g^{-1}g'))) \ox_\IZ \psi_{z,z'}
    \end{equation*}
    for $z = (g,e,t), z'= (g',e',t') \in  G \x E \x \einsu$.
    Thus with respect to the decomposition
    $\IZ^{nm} = \bigoplus_{k=0}^{m-1} T_k$ the $(k,k')$-block of
    $\big( \tr_\alpha (f,\psi) \big)_{z,z'}$ is given by
    \begin{equation*}
     \big( \big( \tr_\alpha (f,\psi) \big)_{z,z'} \big)_{k,k'} =
       \begin{cases} 
         (f \circ \rho(\alpha(g_k^{-1} g^{-1}g' g_{k'}))) \ox_\IZ \psi_{z,z'}
               & \text{if} \; g'g_{k'} g g_{k} \overline{H} = \overline{H}, \\
         0 & \text{otherwise.}
       \end{cases} 
    \end{equation*}
    Comparing this to the definition of $\wt{\tr}_\alpha$ we see that
    $\tau$ is natural for morphisms.
    \\[1mm]~\ref{lem:tilde-tr:control}
    By definition we have for $z = (g,gg_k \overline{H},e,t)$,
    $z' = (g',g'g_{k'} \overline{H},e',t')$
    \begin{equation*}
      \big( \wt{\tr}_\alpha (f,\psi) )_{z,z'} \neq 0 
          \quad \iff \quad \big( gg_k \overline{H} = g' g_{k'} \overline{H} 
            \;  \text{and}  \; \psi_{z,z'} \neq 0 \big)
    \end{equation*}
    where $z = (g,e,t)$, $z' = (g',e',t')$.
   }
  \end{proof}


  \subsection{Surjectivity of the 
        $P_k$ in~\eqref{eq:main-diagram}}

  In the remainder of this section we use the notation 
  from Section~\ref{sec:core}.
  In particular $G$ will from now on be a Farrell-Hsiang group.
  We denote by $(p_n)_* \colon \calo^G(E,S_n,d_{S_n}) \to \calo^G(E,G,d_G)$ 
  the functor induced by the projection 
  $p_n \colon S_n = G \x \coprod_{H \in \calh_n} G/\ol{H} \to G$.  

  \begin{proposition}
    \label{prop:onto}
    Let $n \in \IN$.
    \begin{enumerate}
    \item \label{prop:onto:K}
       There are linear functors
       $F_n^+, F_n^- \colon \calo^G(E,G,d_G) \to \calo^G(E,S_n,d_{S_n})$
       with the following two properties 
       \begin{itemize}
       \item $K_*((p_n)_* \circ F_n^+)  - 
                  K_*((p_n)_* \circ F_n^-)$ is the identity
          on $K_*( \calo^G(E,G,d_G) )$;
       \item if $R > 0$ and $\psi \in \calo^G(E,G,d_G)$ is
          $R$-controlled, then $F_n^+(\psi)$ and $F_n^-(\psi)$
          are both also $R$-controlled. 
       \end{itemize}
    \item \label{prop:onto:L}
       There are functors of additive categories with involutions
       $G_n^+ = (G_n^+,E_n^+)$ and  
       $G_n^- = (G_n^-,E_n^-) 
           \colon \calo^G(E,G,d_G) \to \calo^G(E,S_n,d_{S_n})$
       with the following  properties 
       \begin{itemize}
       \item  $L_*((p_n)_* \circ G_n^+) - L_*((p_n)_* \circ G_n^-)$
          is the identity
          on $L_*( \calo^G(E,G,d_G) )$;
       \item if $R > 0$ and $\psi \in \calo^G(E,G,d_G)$ is
          $R$-controlled, then $G^+_n(\psi)$ 
          and $G_n^-(\psi)$ are both also $R$-controlled.
       \item
          Denote by $I$ both the involution on $\calo^G(E,S_n,d_{S_n})$
          and the involution on $\calo^G(E,G,d_G)$.
          For each object $M \in \calo^G(E,G,d_G)$,
          the isomorphisms 
          $E_n^+(M) \colon G_n^+(I(M)) \to I(G_n^+(M))$
          and   
          $E_n^-(M) \colon G_n^-(I(M)) \to I(G_n^-(M))$
          are $0$-controlled.
       \end{itemize}
    \end{enumerate}
  \end{proposition}
  
  \begin{proof}~\ref{prop:onto:K}
    By Theorem~\ref{thm:Swan_and_Dress}~\ref{thm:Swan_and_Dress:Swan} there 
    are $\tau_H \in \Sw(\IZ,H)$, $H \in \calh_n$ such that
    $1_{\Sw} = \sum_{H \in \calh_n} \ind_H^{F_n}(\tau_H) \in \Sw(\IZ,F_n)$.
    Any element in $\Sw(\IZ,H)$ can be written as the difference 
    of the classes of two modules.
    Pick modules $T^+_H$ and $T^-_H \in \modules_{(\IZ,G)}$, $H \in \calh_n$
    such that $\tau_H = [T^+_H] - [T^-_H]$.
    Because $\res_{\alpha_n}$ sends $1_{\Sw} \in \Sw(\IZ,F_n)$ 
    to $1_{\Sw} \in \Sw(\IZ,G)$ we obtain 
    \begin{equation*}
      1_{\Sw} = \sum_{H \in \calh_n} 
                         [\res_{\alpha_n} \circ \ind_H^{F_n} (T^+_H)]
                       - [\res_{\alpha_n} \circ \ind_H^{F_n} (T^-_H)]
       \in \Sw(\IZ,G)
    \end{equation*}
    For $H \in \calh_n$ we have a canonical inclusion
    $G \x G/\ol{H} \to S_n = G \x \coprod_{K \in \calh_n} G/ \ol{K}$
    that induces an inclusion
    $\calo^G(E,G \x G/\ol{H},d_{G,\ol{H}}) \to \calo^G(E,S_n,d_{S_n})$. 
    Define $F^\pm_{H}$ as the composition of $\wt{\tr}_{\alpha_n}(T^\pm_H,-)$
    with this inclusion.
    Then  $K_*((p_n)_* \circ F^\pm{H}) = 
       K_*(\tr(\res_{{\alpha_n}} \ind_H^{F_n} (T^\pm_H),-))$
    by Lemma~\ref{lem:tilde-tr}~\ref{lem:tilde-tr:equivalent}.
    Define now
    \begin{equation*}
      F_n^\pm := \bigoplus_{H \in \calh_n} F^\pm_H. 
    \end{equation*}
    The functor $\tr$ is exact by Lemma~\ref{lem:tr-is-exact} and so 
    Proposition~\ref{prop:swan-acts-on-B} applies.
    Therefore we can compute for all $a \in K_*(\calo^G(E,G,d_G))$ 
    \begin{eqnarray*}
      \lefteqn{K_*((p_n)_* \circ F_n^+)(a) - K_*((p_n)_* \circ F_n^-)(a)} 
      \\
      & = &
      \sum_{H \in \calh_n} K_*((p_n)_* \circ F^+_H)(a) - 
                             K_*((p_n)_* \circ F^-_H)(a)
      \\
      & = &
      \sum_{H \in \calh_n}  
                K_*(\tr(\res_{{\alpha_n}} \circ \ind_H^{F_n}(T^+_H)))(a)
                 -  K_*(\tr(\res_{{\alpha_n}} \circ\ind_H^{F_n}(T^-_H)))(a)
      \\
      & = &
      \sum_{H \in \calh_n} 
             \mu_{\tr}([\res_{{\alpha_n}} \circ\ind_H^{F_n}(T^+_H)] \ox a)
            - \mu_{\tr}([\res_{{\alpha_n}} \circ\ind_H^{F_n}(T^-_H)] \ox a)
      \\
      & = &
      \mu_{\tr}((\sum_{H \in \calh_n} 
                 [\res_{{\alpha_n}} \circ \ind_H^{F_n}(T^+_H)]
                   - [\res_{{\alpha_n}} \circ \ind_H^{F_n}(T^-_H)]) \ox a)
      \\
      & = &
      \mu_{\tr}( 1_{\Sw} \ox a) = a
    \end{eqnarray*}
    If $R > 0$ and  $\psi \in \calo^G(E,G,d_G)$ is $R$-controlled
    then each $F^\pm_H(\psi)$ is $R$-controlled, because of the 
    control property of $\wt{\tr}_{\alpha_n}$ 
    (Lemma~\ref{lem:tilde-tr}~\ref{lem:tilde-tr:control}) and because
    $G \x G/\ol{H} \to S_n$ is an isometric embedding.
    The direct sum of $R$-controlled morphisms is again
    $R$-controlled and therefore $F^+(\psi)$ and $F^-(\psi)$
    are both $R$-controlled.  
    \\[1mm]~\ref{thm:onto:L}
    We can proceed exactly as in the $K$-theory case.
    By Theorem~\ref{thm:Swan_and_Dress}~\ref{thm:Swan_and_Dress:Dress}
    there are $\sigma_H \in \GW(\IZ,H)$, $H \in \calh_n$ 
    such that $1_{\GW} = \sum_{H \in \calh_n} \ind_{F_n}^{H}(\sigma_H)$.
    Any element in $\GW(\IZ,H)$ can be written as the difference 
    of the classes of two symmetric forms.
    Pick symmetric forms $(T^+_H,\varphi^+_H)$ and 
    $(T^-_H,\varphi^-_H)$ over $\modules_{(\IZ,G)}$, $H \in \calh_n$
    such that $\sigma_H = [(T^+_H,\varphi^+_H)] - [(T^-_H,\varphi^-_H)]$.
    Define $G^\pm_H$ as the composition of
    $\tr_{\alpha_n}((T^+_H,\varphi^+_H),-)$ with the inclusion
    $\iota_H \colon 
      \calo^G(E,G \x G/\ol{H},d_{G,\ol{H}}) \to \calo^G(E,S_n,d_{S_n})$
    and set
    \begin{equation*}
      G_n^\pm := \bigoplus_{H \in \calh_n} G_H^\pm.
    \end{equation*}
    As in the $K$-theory  case it follows 
    (using now Proposition~\ref{prop:GW-acts-on-B}) that for 
    all $b \in L_*(\calo^G(E,G,d_G))$ we have
    \begin{equation*}
      L_*((p_n)_* \circ G_n^+) (b) - L_*((p_n)_* \circ G_n^-)(b) = b
    \end{equation*}
    and that $G_n^\pm(\psi)$ is $R$-controlled, whenever
    $\psi$ is $R$-controlled.

    It remains to prove the final claim.
    Let $M$ be an object from $\calo^G(E,G,d_G)$.
    Then
    \begin{eqnarray*}
       G_n^\pm(I(M)) = \bigoplus_H G_H^+(I(M)) 
       & = & 
       \bigoplus_H \iota_H(\wt{\tr}_{\alpha_n}(T_H^\pm,I(M)) \\
       I(G_n^\pm(M) = \bigoplus_H I (G_H^+(M)) 
       & = & 
       \bigoplus_H I (\iota_H(\wt{\tr}_{\alpha_n}(T_H^\pm,M) \\
       & = & \bigoplus_H \iota_H(\wt{\tr}_{\alpha_n}((T_H^\pm)^*,I(M))
    \end{eqnarray*}
    and 
    \begin{equation*}
      E^\pm_n(M) = 
         \bigoplus_H \iota_H(\wt{\tr}_{\alpha_n}(\varphi_H^\pm,\id_{I(M)}). 
    \end{equation*}
    The control claim follows from 
    Lemma~\ref{lem:tilde-tr}~\ref{lem:tilde-tr:control} because 
    $\id_{I(M)}$ is $0$-controlled.
  \end{proof}

  \begin{theorem} 
     \label{thm:onto}
     \
     \begin{enumerate}
     \item \label{thm:onto:K}
        For all $a \in K_* (\calo^G(E,G,d_G))$ there is 
        $\hat a \in K_* (\calo^G(E,(S_n,d_{S_n})_{n \in \IN})$  
        such that for all $k$ we have $(K_*(P_k))(\hat a) = a$.
     \item \label{thm:onto:L}
        For all $b \in L_* (\calo^G(E,G,d_G))$ there is 
        $\hat b \in L_* (\calo^G(E,(S_n,d_{S_n})_{n \in \IN})$  
        such that for all $k$ we have $(L_*(P_k))(\hat b) = b$.
     \end{enumerate}
  \end{theorem}

  \begin{proof}~\ref{thm:onto:K}
    Let $F_n^+$, $F_n^-$ be the sequences of functors from
    Proposition~\ref{prop:onto}~\ref{prop:onto:K}.
    Because of the control property in~\ref{prop:onto}~\ref{prop:onto:K}
    the product functors
    \begin{equation*}
      \prod_n F_n^+ \colon 
      \calo^G(E,G,d_G) \to \prod_n \calo^G(E,S_n,d_{S_n})
    \end{equation*}
    lift uniquely to functors
    \begin{equation*}
      F^\pm \colon \calo^G(E,G,d_G) \to \calo^G(E,(S_n,d_{S_n})_{n \in \IN}).
    \end{equation*}
    Then $P_k \circ F^\pm = (p_k)_* \circ F^\pm_k$ for all $k \in \IN$.
    Thus the first assertion in~\ref{prop:onto}~\ref{prop:onto:K}
    implies that
    $K_*(P_k)  (K_*(F^+)(a) - K_*(F^-)(a)) = a$ for
    all $a \in \calo^G(E,G,d_G)$.
    Therefore we can set $\hat a := K_*(F^+)(a) - K_*(F^-)(a)$.
    \\[1mm]~\ref{thm:onto:L}
    For $L$-theory we can argue exactly as we did for $K$-theory,
    now using the $G_n^\pm$ from 
    Proposition~\ref{prop:onto}~\ref{prop:onto:L}.
    Here the third assertion in~\ref{prop:onto}~\ref{prop:onto:L}  
    is needed to ensure that the $E_n^\pm$ can be combined to 
    a natural transformation, just as the second assertion is 
    needed to ensure that the $G_n^\pm$ can be combined to a functor. 
  \end{proof}


\typeout{-------------------- Section 7: The functor F ----------------------}  
  \section{The functor $F$}
    \label{sec:F}

  We use the notation from Section~\ref{sec:core}.
  Note first that for any subgroup $U$ of $G$ 
  there is a bijection of $G$-sets
  $G \x G/U \to 
   \ind_G^{U} \res_U^G G = G \x_{U} G$ 
  defined by $(a, g U) \mapsto (g,g^{-1}a)$;
  the inverse is given by $(g,b) \mapsto gb, g U$.
  (We use the diagonal $G$-action on $G \x G/U$.)
  For $H \in \calh_n$ we obtain a $G$-map 
  $\tilde{f}_H \colon G \x G/\ol{H} \to \ind_G^{\ol{H}} E_H$
  by composing this bijection (for $U = \ol{H}$) with 
  $\ind_G^{\ol{H}} f_H \colon 
     \ind_G^{\ol{H}} G \to
     \ind_G^{\ol{H}} E_H$.
  Define the $G$-map  $f_n \colon S_n \to X_n$ by
  \begin{equation*}
    f_n( a, g \ol{H}) := (a, \tilde{f}_H(a,g \ol{H}) 
         = (a, g, f_H(g^{-1}a)),  
  \end{equation*}
  for $a, g \in G$ and  $H \in \calh_n$. 

  \begin{proposition}
    \label{prop:F}
    The sequence of maps $(f_n)_{n \in \IN}$ induces a functor
    \begin{equation*}
      F \colon \calo^G(E, (S_n, d_{S_n})_{n \in \IN})
       \to \calo^G(E, (X_n, d_{X_n})_{n \in \IN}).
    \end{equation*}
    For all $k$ we have $q_k \circ F = p_k$.  
  \end{proposition}

  \begin{proof}
    We need to show that the sequence $(f_n)_{n \in \IN}$
    is compatible with the metric control conditions
    for the sequences of quasi-metrics $(d_{S_n})_{n \in \IN}$
    and $(d_{X_n})_{n \in \IN}$ more precisely we need to show
    that for any $r \in (0,\infty)$ there is $R \in (0,\infty)$ such that
    for all $n$ and $s ,s' \in S_n$ the implication
    \begin{equation} \label{eq:r-R}
      d_{S_n}(s,s') < r \quad \implies \quad 
       d_{X_n}(f_n(s), f_n(s')) < R
    \end{equation}
    holds.
    
    Let $r \in (0,\infty)$ be given.
    The $G$-action on $S_n$ is cofinite, the quasi-metrics
    $d_{S_n}$ and $d_{G_n}$ are $G$-invariant and $f_n$ is $G$-equivariant.
    For each $s \in S_n$ there are only finitely many $s' \in S_n$
    such that $d_{S_n}(s',s) < r$, because the word metric $d_G$ has this 
    property on $G$.
    This implies that 
    $D_r := \{ d_{X_n}(f_n(s), f_n(s')) \mid n < r, s,s' \in S_n, 
                                                d_{S_n}(s,s') < r \}$
    is a finite set.
    We can therefore define $R := 1 + r + \max D_r$.
    We claim that then~\eqref{eq:r-R} holds for all $n$ and all
    $s, s' \in S_n$.
    If $n < r$, then this is clear from the definition of $R$.
    Let $n > r$ and $s,s' \in S_n$ with $d_{S_n}(s,s') < r$.
    Write $s = (a, g\ol{H})$ and $s'= (a',g'\ol{H'})$
    with $H, H' \in \calh_n$, $a,a',g,g' \in G$.
    Since $d_{S_n}(s,s') < r < \infty$ it follows from the
    definition of $d_{S_n}$ that $H = H'$, $g \ol{H} = g' \ol{H'}$
    and $d_G(a,a') < r < n$.
    Since $d_G$ is $G$-invariant we also have
    $d_G(g^{-1}a,g^{-1}a') < n$.
    We conclude from the crucial contracting property of $f_H$
    that $d^1_{E_H}(f_H(g^{-1}a),f_H(g^{-1}a')) < \frac{1}{n}$. 
    Since $s = (a,g\ol{H})$, $s' = (a',g\ol{H})$ we have
    $f_n(s) = (a,g,f_H(g^{-1}a))$, $f_n(s') = (a',g,f_H(g^{-1}a'))$.
    Thus
    \begin{eqnarray*}
      d_{X_n}(f_n(s),f_n(s')) 
       & = & d_{X_n}((a,g,f_H(g^{-1}a)),(a',g,f_H(g^{-1}a'))) \\
       & = & d_G(a,a') + n \cdot
          d^1_{\ind_G^{\ol{H}} E_H}( (g,f_H(g^{-1}a)),(g,f_H(g^{-1}a'))) \\
       & = & d_G(a,a') + n \cdot d^1_{E_H}(f_H(g^{-1}a),f_H(g^{-1}a')) \\
       & < & r + n \cdot \frac{1}{n} = r + 1 < R. 
    \end{eqnarray*}
    This proves our claim.
    Thus $(f_n)_{n \in \IN}$ induces a functor $F$.

    For the canonical projections  $p_k \colon S_k  \to G$ and
    $q_k \colon X_k \to G$ we have $q_k \circ f_k = p_k$.
    This implies that $Q_k \circ F = P_k$.      
  \end{proof}


  \appendix

  \section{Applications and examples of 
                    Farrell-Hsiang groups}
  
  Proofs of the Farrell-Jones Conjecture often combine methods 
  from controlled topology (for example our Theorem~\ref{the:FH_implies_FJ})
  with group theoretic and geometric considerations (for example to show
  that certain groups are Farrell-Hsiang groups with respect to
  some family $\calf$) and an induction using the transitivity 
  principle~\cite[Theorem~A.10]{Farrell-Jones(1993a)}.
  The transitivity principle asserts that for families of groups 
  $\calf \subseteq \calg$ the Farrell-Jones Conjecture for $G$
  holds relative to $\calf$ provided a) the Farrell-Jones Conjecture
  for $G$ holds relative to $\calg$ and b) for any $H \in \calg$
  the Farrell-Jones Conjecture holds relative to $\calf$.
  In the following we briefly discuss some results 
  from~\cite{Bartels-Farrell-Lueck(2011)} and their
  connection to Farrell-Hsiang groups.
  
  Many crystallographic groups are Farrell-Hsiang groups
  relative to interesting families of subgroups, 
  see~\cite[Proofs of Lemma~2.8, Lemma~2.15,
    Theorem~2.1]{Bartels-Farrell-Lueck(2011)}.
  For example  $\IZ^2 \rtimes_{-\id} \IZ/2$ 
  is a Farrell-Hsiang group relative to $\VCyc$.   
  In combination with the transitivity principle this yields a 
  proof of the Farrell-Jones Conjecture  with additive categories as coefficients
  for virtually finitely generated   abelian groups.
  This generalizes~\cite{Quinn(2005)} where only untwisted ring as coefficients are treated.
  (The version with additive categories as coefficients has better inheritance
   and transitivity properties and encompasses the so called fibered version).

 The main motivation for this paper is that its methods apply to situations,
  where the known techniques for 
  virtually abelian groups do not work anymore.
  Namely, special affine groups are Farrell-Hsiang groups relative to
  the family of virtually finitely generated abelian groups,
  see~\cite[Proof of Proposition~3.40]{Bartels-Farrell-Lueck(2011)}. 
  This fact is a key ingredient for the proof of the 
  Farrell-Jones Conjecture  with additive categories as coefficients  
  for virtually poly-cyclic groups and finally 
  for cocompact lattices in virtually connected Lie groups
  in~\cite{Bartels-Farrell-Lueck(2011)}. 

  In summary, our axiomatic treatment of the Farrell-Hsiang method
  in Theorem~\ref{the:FH_implies_FJ} encapsulates completely the input
  of controlled topology to~\cite{Bartels-Farrell-Lueck(2011)},
  separates it from the necessary group theoretic and geometric
  arguments carried out there, and applies for instance 
  to special affine groups.


\def\cprime{$'$} \def\polhk#1{\setbox0=\hbox{#1}{\ooalign{\hidewidth
  \lower1.5ex\hbox{`}\hidewidth\crcr\unhbox0}}}

  \addcontentsline{toc<<}{section}{References} 

\end{document}